\documentclass[12pt,oneside,reqno]{amsart}
\usepackage{amsmath,amsthm,amssymb}
\usepackage[pagewise]{lineno}
\usepackage[utf8]{inputenc}
\usepackage[T1]{fontenc}
\usepackage{hyperref}
\usepackage{enumerate}
\usepackage{bbm}
\usepackage{todonotes}
\usepackage{mathrsfs}
\usepackage{color}
\usepackage{nicefrac}
\usepackage{fancyhdr}
\usepackage{centernot}
\usepackage[a4paper]{geometry}
\usepackage{cleveref}

\usepackage[sort&compress,numbers]{natbib}
\newtheorem{Th}{Theorem}[section]
\newtheorem{Prop}[Th]{Proposition}
\newtheorem{Lemma}[Th]{Lemma}
\newtheorem{Cor}[Th]{Corollary}
\theoremstyle{definition}
\newtheorem{Remark}[Th]{Remark}
\newtheorem{Def}[Th]{Definition}
\newtheorem{Example}[Th]{Example}
\newtheorem{Question}[Th]{Question}
\makeatletter
\newcommand{\customlabel}[2]{%
   \protected@write \@auxout {}{\string \newlabel {#1}{{#2}{\thepage}{#2}{#1}{}} }%
   \hypertarget{#1}{#2}
}
\makeatother
\newcommand{\divides}{\mid}
\newcommand{\ndivides}{\nmid}
\newcommand{\n}{\mathbf{n}}

\newcommand{\cf}{{\mathcal{F}}}

\newcommand{\Q}{\mathbb{Q}}
\newcommand{\R}{{\mathbb{R}}}

\newcommand{\C}{{\mathbb{C}}}
\newcommand{\Z}{{\mathbb{Z}}}
\newcommand{\N}{{\mathbb{N}}}

\newcommand{\HH}{\mathbb{H}}
\newcommand{\G}{\mathbb{G}}
\newcommand{\vep}{\varepsilon}

\newcommand{\sB}{\mathscr{B}}
\newcommand{\raz}{\mathbbm{1}}

\newcommand{\OK}{\mathcal{O}_K}

\makeatletter
\def\moverlay{\mathpalette\mov@rlay}
\def\mov@rlay#1#2{\leavevmode\vtop{%
   \baselineskip\z@skip \lineskiplimit-\maxdimen
   \ialign{\hfil$\m@th#1##$\hfil\cr#2\crcr}}}
\newcommand{\charfusion}[3][\mathord]{
    #1{\ifx#1\mathop\vphantom{#2}\fi
        \mathpalette\mov@rlay{#2\cr#3}
      }
    \ifx#1\mathop\expandafter\displaylimits\fi}
\makeatother

\begin{document}
\title{Proximality of multidimensional $\mathscr{B}$-free systems}
\author{Aurelia Dymek}
\begin{abstract}We characterize proximality of multidimensional $\mathscr{B}$-free systems in the case of number fields and lattices in $\Z^m$, $m\geq2$.
\end{abstract}
\maketitle
\section{Introduction}
Let $(T_g)_{g\in\mathbb{G}}$ be an action of a~countable Abelian group $\mathbb{G}$ by homeomorphisms on a~compact metric space $(X,D)$. The pair $(X,(T_g)_{g\in\G})$ is called a \emph{topological dynamical system}. Two, mutually complementary, basic concepts of topological dynamics are distality and proximality. Recall that a pair $(x,y)$ of two different points from $X$ is called \emph{distal} if $\liminf_{g\to\infty}D(T_gx,T_gy)>0$, otherwise $(x,y)$ is called \emph{proximal}. If any pair of distinct points in $X$ is distal (respectively, proximal) then $(X,(T_g)_{g\in\G})$ is called {\em distal} (respectively, \emph{proximal}). 
In the minimal case, distality is rather well understood by the structural result of Furstenberg \cite{Fu}. While in general distal systems can display complicated dynamics, we are interested only in subshifts, i.e., closed subsets $X\subseteq \{0,1\}^{\G}$ invariant under the action by shifts $T_h((x_g)_{g\in\G})=(x_{g+h})_{g\in\G}$, $h\in \G$. Since any two points $x,y$ whose all shifts remain close must be equal, it follows immediately that any distal subshift is finite.
On the other hand, unless $(X,(T_g)_{g\in\G})$ is minimal (see Lemma \ref{fixed} below), the proximality of subshifts is far from being understood, for some results see \cite{MR3182747,MR3114317}. 
 
Denote by $Prox\subseteq X\times X$ the set of all proximal pairs.  The relation $Prox$ is reflexive, symmetric, $(T_g)_{g\in\mathbb{G}}$-invariant, but, in general, is not transitive. In order to obtain an equivalence relation, a stronger notion than proximality is needed. The pair $(x,y)\in X\times X$ is called {\em syndetically proximal} if for any $\vep>0$ the set $\{g\in\G\ : \ D(T_gx,T_gy)<\vep\}$ is syndetic and we write $(x,y)\in SynProx$. Recall that a subset $A\subseteq\G$ is syndetic if there exists a finite subset $F\subseteq\G$ such that $A+F:=\{a+f\ : \ a\in A, f\in F\}=\G$. Clay~\cite{Cl} proved that $SynProx$ is an equivalence relation and Wu \cite{Wu} showed that if $Prox$ is transitive then $Prox=SynProx$. So, $Prox$ is an equivalence relation if and only if $Prox=SynProx$.
It follows that $(X,(T_g)_{g\in\mathbb{G}})$ is proximal if and only if $(X,(T_g)_{g\in\mathbb{G}})$ is syndetically proximal.

In this paper, we study proximality of generalizations of $\mathscr{B}$-free systems \cite{Ab-Le-Ru, AD-SK-JKP-ML, MR3356811, JKP-ML-BW, JK-MK-DK, Me}. Let $\mathscr{B}\subseteq\N$. Integers with no factors in $\mathscr{B}$ are called \emph{$\mathscr{B}$-free numbers} and are denoted by $\mathcal{F}_\mathscr{B}$. Such sets were studied already in the 30's by Behrend, Chowla, Davenport, Erd\H{o}s, Schur and others, see \cite{MR1414678}. Note that, if $\mathscr{S} =\{p^2 : p\text{ is prime}\}$ then $\raz_{\mathcal{F}_\mathscr{S}}=\mu^2$, where $\mu\colon\Z\to\C$ is the Möbius function given by the following formula:
\begin{equation*}
\mu(n)=\begin{cases}
1, \text{ if }|n|=1,\\
(-1)^t, \text{ if } |n| \text{ is the product of }t \text{ distinct primes},\\
0, \text{ otherwise.}
\end{cases}
\end{equation*}
The dynamical approach to study $\mathscr{B}$-free systems is rather new.
\begin{enumerate}[(I)]
\item\label{(I)} $\mathscr{B}$-free systems
 
In 2010, Sarnak in his seminal paper \cite{sarnak-lectures} proposed to study the dynamical systems determined by $\mu$ and $\mu^2$. In either case, we consider the closure $X_\eta$ of the orbit of $\eta=\mu\in\{-1,0,1\}^\Z$ or $\eta=\mu^2\in\{0,1\}^\Z$ under the left shift $S$.  
The dynamics of $(X_\mu,S)$ is complicated and there are many open questions related to it, see, e.g., \cite{SF-JKP-ML}. The system $(X_{\mu^2},S)$ (called \emph{square-free system}) which is a topological factor of $(X_\mu,S)$ via the map $(x_n)_{n\in\Z}\mapsto(x_n^2)_{n\in\Z}$ is simpler to study.
Similarly, given $\sB\subseteq\N$, taking the closure of the orbit of $\eta=\raz_{\cf_\sB}\in\{0,1\}^\Z$ under the left shift, yields a $\sB$-free system. At first, $\mathscr{B}$-free systems were studied in the \emph{Erd\H{o}s case}, i.e., for $\mathscr{B}$ infinite, pairwise coprime, with $\sum_{b\in\mathscr{B}}\frac{1}{b}<\infty$ \cite{Ab-Le-Ru,MR3356811, JKP-ML-BW, JK-MK-DK, Me}. Theorem 8 in \cite{sarnak-lectures} gives proximality of the square-free system, cf. also \cite{Ab-Le-Ru}.
The general case is considered in~\cite{AD-SK-JKP-ML} and the following characterization of the proximality is given:
\begin{Th}[{\cite[Theorem 3.7]{AD-SK-JKP-ML}}]\label{one-dim}
 Let $\mathscr{B}\subseteq\N\setminus\{1\}$ and $\eta=\raz_{\mathcal{F}_\mathscr{B}}$. The following conditions are equivalent:
\begin{enumerate}[(a)]
\item $(X_\eta, S)$ is proximal,\label{(a)}
\item $\ldots0.00\ldots\in X_\eta$,
\item for any choice of $q_1,\ldots, q_s > 1$, $s\geq1$, we have $\mathscr{B}\not\subseteq\bigcup_{i=1}^s q_i\Z$,\label{(c)}
\end{enumerate}
\customlabel{(d)}{(d)} $\mathscr{B}$ contains an infinite subset of pairwise coprime integers,
\begin{enumerate}[(a)]
\setcounter{enumii}{4}
\item $\mathcal{F}_\mathscr{B}$ does not contain an infinite arithmetic progression.
\end{enumerate}
\end{Th}
Recently, Kasjan, Keller and Lemańczyk considered $\mathscr{B}$-free systems as weak model sets \cite{SK-GK-ML} and characterized the proximality of $\mathscr{B}$-free systems by a property of associated window, see also \cite{GK}.
\end{enumerate}
Since $b\Z$ (for $b\in\Z$) is simultaneously a lattice and an ideal in $\Z$, one dimensional Erd\H{o}s case has two natural generalizations:
\begin{enumerate}[(I)]
\setcounter{enumi}{1}
\item\label{(II)} $\mathscr{B}$-free systems in lattices

Baake and Huck in their survey \cite{Baake:2015aa} define $\mathscr{B}$-free lattice points in a lattice $\Lambda\subseteq\R^m$, $m\geq2$ by the formula $\mathcal{F}_\mathscr{B}:=\Lambda\setminus\bigcup_{b\in\mathscr{B}}b\Lambda$,
where $\mathscr{B}\subseteq\N$ is an infinite pairwise coprime set. For $\eta=\raz_{\mathcal{F}_\mathscr{B}}\in\{0,1\}^{\Lambda}$ they consider its orbit closure $X_\eta$ under the multidimensional shift $(S_\lambda)_{\lambda\in\Lambda}$. The system $(X_\eta,(S_\lambda)_{\lambda\in\Lambda})$ is called a \emph{$\mathscr{B}$-free system}. They prove that these $\mathscr{B}$-free systems are proximal.
\item\label{(III)} $\mathfrak{B}$-free systems in number fields

Baake and Huck \cite{Baake:2015aa} also define $\mathfrak{B}$-free integers in number fields which generalizes the case studied by Cellarosi and Vinogradov \cite{FC-IV} and \eqref{(II)}. Given a finite
extension $K$ of $\Q$, with the ring of integers $\mathcal{O}_K$, they set $\mathcal{F}_\mathfrak{B}:= \mathcal{O}_K\setminus\bigcup_{\mathfrak{b}\in\mathfrak{B}}\mathfrak{b}$,
where $\mathfrak{B}$ is an infinite pairwise coprime collection of ideals in $\mathcal{O}_K$ with $\sum_{\mathfrak{b}\in\mathfrak{B}}
\frac{1}{|\mathcal{O}_K/\mathfrak{b}|}<\infty$. Similarly as above, for $\eta=\raz_{\mathcal{F}_\mathfrak{B}}\in\{0,1\}^{\mathcal{O}_K}$ they define its orbit closure $X_\eta$ under the multidimensional shift $(S_a)_{a\in\mathcal{O}_K}$, where $\OK$ is considered as an additive group. They call the system $(X_\eta,(S_a)_{a\in\mathcal{O}_K})$ a \emph{$\mathfrak{B}$-free system} and announce similar results as for $\mathscr{B}$-free systems in lattice case (leaving the details to the reader). 
By the existence of a group isomorphism between $\OK$ and $\Z^{[K:\Q]}$ (known as Minkowski embedding), we have that \eqref{(III)} covers \eqref{(II)}.
\end{enumerate}

While the number of lattices with index less than $x$ in $\Z^2$ grows quadratically \cite[p. 10]{MR1460018}, the number of ideals with norm less than $x$ grows only linearly \cite[Proposition 2.1]{MR3821719}. Thus, not all lattices in $\Z^{[K:\Q]}$ are images of ideals in $\OK$  by the Minkowski embedding. A natural question arises: 
\begin{Question}
Is there an analog of Theorem \ref{one-dim} in case of lattices and number fields?
\end{Question}
Clearly, in both settings we drop the assumption of pairwise coprimeness, going beyond the Erd\H os case. In comparison to lattices, the integer ring carries an additional multiplicative structure which allows us to give the positive answer to our question in case of number fields. In fact, we prove more. Let $m\geq1$, let $\mathfrak{B}$ be a collection of ideals with finite indices in $\OK^m=\underbrace{\OK\times\ldots\times\OK}_m$, $\mathcal{M}_\mathfrak{B}:=\bigcup_{\mathfrak{b}\in\mathfrak{B}}\mathfrak{b}$ and $\mathcal{F}_\mathfrak{B}:=\OK^m\setminus\mathcal{M}_\mathfrak{B}$. We consider the orbit closure $X_\eta$ of $\eta=\raz_{\cf_\mathfrak{B}}\in\{0,1\}^{\OK^m}$ under the multidimensional shift $(S_a)_{a\in\OK^m}$. Our main result is the following:
\begin{Th}\label{ideals_general}
Let $(X_\eta,(S_a)_{a\in\OK^m})$ be as above.
The following conditions are equivalent:
\begin{enumerate}[(a)]
\item $(X_\eta,(S_a)_{a\in\OK^m})$ is proximal,\label{giC}
\item $\mathbf{0}\in X_\eta$, where $\mathbf{0}_a=0$ for any $a\in\mathcal{O}_K^m$,\label{giD}
\item for any proper ideals $I_1,I_2,\ldots,I_k$, $k\geq1$, with finite indices in $\OK^m$, we have $\mathcal{M}_\mathfrak{B}\not\subseteq\bigcup_{j=1}^k I_j$,\label{giF}
\item $\mathfrak{B}$ contains an infinite collection of pairwise coprime ideals,\label{giG}
\item for any $a\in\OK^m$ and any ideal $I$ with finite index, we have $I+a\not\subseteq\mathcal{F}_\mathfrak{B}$,\label{giH}
\item $d^*(\mathcal{M}_\mathfrak{B})=1$, where $d^*$ denotes the upper Banach density (see Definition \ref{Banach}).\label{gdens}
\end{enumerate}
\end{Th}
In particular, if $m=1$ and $\OK=\Z$ then we recover Theorem \ref{one-dim}.

In case of lattices (as in \eqref{(II)}), the analogue of the implication \eqref{giC}$\implies$\eqref{giG} may fail (see Examples \ref{ex2} and \ref{ex1}). All other conditions remain equivalent (with some necessary modification in \eqref{giF}), for the detailed formulation, see Theorem \ref{general_lattices}.

In order to obtain an analogue of \eqref{giC}$\implies$\eqref{giG}, we assume that the lattices are of a special form: $\Lambda_i=a_1^{(i)}\Z\times a_2^{(i)}\Z\times\ldots\times a_m^{(i)}\Z$, where $(a_1^{(i)},\ldots,a_m^{(i)})\in\N^m\setminus\{(1,1,\ldots,1)\}$. Then, in fact each $\Lambda_i\subseteq\Z^m$ is an ideal of finite index and as an immediate consequence of Theorem~\ref{ideals_general}, we have:
\begin{Cor}\label{rectangular} Suppose that $\Lambda_i=a_1^{(i)}\Z\times a_2^{(i)}\Z\times\ldots\times a_m^{(i)}\Z$, where $(a_1^{(i)},\ldots,a_m^{(i)})\in\N^m\setminus\{(1,1,\ldots,1)\}$, $i\geq1$. Then $(X_\eta,(S_{\mathbf{n}})_{\mathbf{n}\in\Z^m})$ is proximal if and only if $\{\Lambda_i\}_{i\geq1}$ contains an infinite pairwise coprime subset.
\end{Cor}
\paragraph{Acknowledgment} The author obtained financial resources for the preparation of a doctoral dissertation from the National Science Center as part of the financing of a doctoral scholarship based on the decision number 2018/28/T/ST1/00435. The author thanks Mariusz Lema{{\'n}}czyk and Joanna Kułaga-Przymus for helpful discussions and remarks on this manuscript, Stanisław Kasjan for all remarks on the lattice case and Vitaly Bergelson for properties of the upper Banach density for amenable groups.
\section{Preliminaries}
\subsection{Subshifts}
Given a countable Abelian group $\mathbb{G}$ and a finite alphabet $\mathcal{A}$, there is a natural action of $\mathbb{G}$ on $\mathcal{A}^\mathbb{G}$ by (commuting) translations:
\begin{equation}\label{eq:fo:2}
S_g({(x_h)}_{h\in \mathbb{G}})={(y_h)}_{h\in \mathbb{G}},\ y_h=x_{h+g}\text{ for }g,h\in \mathbb{G}.
\end{equation}
We say that ${(F_n)}_{n\geq 1}\subseteq \mathbb{G}$ is a \emph{F\o lner sequence} in $\mathbb{G}$ if $F_n$ is finite for any $n\geq1$ and $\lim_{n\to \infty}\frac{|(F_n+g)\cap F_n|}{|F_n|}=1$
for each $g\in \mathbb{G}$. If additionally $\bigcup_{n\geq1} F_n=\G$ and $F_n\subseteq F_{n+1}$ for each $n\geq 1$, we say that $(F_n)_{n\geq 1}$ is \emph{nested}. For any countable Abelian group there exists a nested F\o lner sequence, see \cite{Folner, MR0079220}. The sequence $(\{-n,\ldots,n\}^d)_{n\geq1}$ is an example of a nested F\o lner sequence in $\Z^d$.
\begin{Remark}\label{topology}
Notice that the product topology on $\mathcal{A}^{\G}$ is metrizable. Let ${(F_n)}_{n\geq 1}$ be a nested F\o{}lner sequence. In any metric inducing the product topology, we have the following characterization of convergence of a sequence $(x^{(s)})_{s\geq 1}$ to $x$ in $\mathcal{A}^{\G}$:
$$
x^{(s)}\to x \iff \forall_{n\geq 1}\exists_{s_n} \forall_{s>s_n} \forall_{g\in F_n} \ x^{(s)}_g=x_g.
$$
In particular, this happens for $D$ given by
\begin{equation}\label{metric}
D(x,y)=\min\left\{1,2^{-\sup\{n\geq 1 : \ x_g=y_g \text{ for each }g\in F_n\}}\right\},
\end{equation}
where we put $2^{-\infty}=0$ and $\sup\emptyset=-\infty$.
\end{Remark}
If $X\subseteq \mathcal{A}^{\G}$ is closed and ${(S_g)}_{g\in\G}$-invariant, we say that $X$ is a \emph{subshift}. We will mostly deal with $\mathcal{A}=\{0,1\}$ and $\G=\OK^m$, where $K$ is an algebraic number field. 
\subsection{Ideals in number fields}

Let $K$ be an algebraic number field of degree $d=[K:\Q]$ with the integer ring $\OK$. As in every Dedekind domain, all proper non-zero ideals in $\OK$ factor (uniquely, up to the order) into a product of prime ideals. 
We will denote ideals in $\mathcal{O}_K$ by $\mathfrak{a},\mathfrak{b},\dots$. We have
$$
\mathfrak{a}+\mathfrak{b}:=\{a+b : a\in\mathfrak{a},b\in\mathfrak{b}\},\ \mathfrak{a}\mathfrak{b}:=\{a_1b_1+\dots+a_kb_k : a_i\in \mathfrak{a},b_i\in\mathfrak{b},1\leq i\leq k, k\geq1\}.
$$
Note that 
\begin{equation}\label{prod*}
\mathfrak{a}\mathfrak{b}\subseteq\mathfrak{a}\cap\mathfrak{b}.
\end{equation}
The \emph{algebraic norm} of an ideal $\mathfrak{a}\neq \{0\}$ is defined as $N(\mathfrak{a}):=|\mathcal{O}_K/\mathfrak{a}|=[\mathcal{O}_K:\mathfrak{a}]$. It is finite and $N(\mathfrak{a}\mathfrak{b})=N(\mathfrak{a})N(\mathfrak{b})$ for any ideals $\mathfrak{a}, \mathfrak{b}\neq\{0\}$ (see e.g.\ \cite[Chapter I, \S 6]{MR1697859}). Finally, recall that there is a natural group isomorphism from $\OK$ to a lattice in $\R^d$, called the \emph{Minkowski embedding} (see e.g.\ \cite[Chapter I, \S 5]{MR1697859}). Thus, $\OK$ is isomorphic to $\Z^d$ as an additive group.

Consider now $\OK^m$, $m\geq1$, as a product ring. Recall that for an infinite collection $\mathfrak{B}$ of ideals with finite indices in $\OK^m$, $\mathcal{M}_\mathfrak{B}=\bigcup_{\ell\geq1}\mathfrak{b}_\ell$ and $\mathcal{F}_\mathfrak{B}=\OK^m\setminus\mathcal{M}_\mathfrak{B}$ is the corresponding set of $\mathfrak{B}$-free numbers. Moreover, $\eta=\raz_{\mathcal{F}_\mathfrak{B}}\in\{0,1\}^{\OK^m}$, i.e.,
\begin{equation}\label{Z:2}
\eta(a)=\begin{cases}
1,& \text{if }a \text{ is }\mathfrak{B}\text{-free},\\
0,& \text{otherwise}
\end{cases}
\end{equation}
for $a\in\OK^m$.

Ideals in $\OK^m$ take a special form:
$$I=I_1\times\ldots\times I_m\subseteq\OK^m,$$
where each $I_j, 1\leq j\leq m$, is an ideal in $\OK$. Indeed, we have:
\begin{Lemma}[see {\cite[Chapter I, \S8, Proposition 8]{MR1727844}}]\label{ideals_in_product}
Let $R_i$ be commutative rings with unity, $i=1,\ldots, m$. Let $I\subseteq R_1\times\ldots\times R_m$ be an ideal. Then there exist ideals $I_1\subseteq R_1,\ldots, I_m\subseteq R_m$ such that $I=I_1\times\ldots\times I_m$.
\end{Lemma}
Moreover, the index of $I=I_1\times\ldots\times I_m\subseteq\OK^m$ is infinite precisely when $I_i=\{0\}$ for some $1\leq i\leq m$. In particular, if $m\geq2$, there exist non-zero ideals of infinite index.

Notice that for $m\geq2$, $\OK^m$ is no longer a domain (all non-zero elements with at least one coordinate equal to zero are zero divisors in $\OK^m$). However, any proper ideal with finite index in $\OK^m$ factors (uniquely, up to the order) into a product of prime ideals. Prime ideals in $\OK^m$ are of the form $\OK^s\times\mathfrak{p}\times\OK^{m-s-1}$, where $\mathfrak{p}$ is a prime ideal in $\OK$ and $0\leq s<m$.
\subsection{Lattices in $\Z^m$}
Let $m\geq1$. We say that a subset $\Lambda\subseteq\Z^m$ is a \emph{lattice} if $\Lambda$ is a subgroup of $\Z^m$ with finite index, i.e., $[\Z^m:\Lambda]<\infty$. For an infinite collection $\mathscr{B}=\{\Lambda_i\}_{i\geq1}$ of lattices, let $\mathcal{M}_\mathscr{B}:=\bigcup_{i\geq1}\Lambda_i$ and $\mathcal{F}_\mathscr{B}:=\Z^m\setminus\mathcal{M}_\mathscr{B}$ be the corresponding set of $\sB$-free lattice points. Let $\eta:=\raz_{\mathcal{F}_\sB}\in\{0,1\}^{\Z^m}$, i.e.,
\begin{equation*}
\eta(\mathbf{m})=\begin{cases}
1,& \text{if }\mathbf{m} \text{ is a }\mathscr{B}\text{-free}\text{ lattice point},\\
0,& \text{otherwise}
\end{cases}
\end{equation*}
for $\mathbf{m}\in\Z^m$. Finally, let $X_\eta\subseteq\{0,1\}^{\Z^m}$ be the orbit closure under the multidimensional shift $(S_{\mathbf{n}})_{\mathbf{n}\in\Z^m}$. We call $(X_\eta,(S_{\mathbf{n}})_{\mathbf{n}\in\Z^m})$ a \emph{$\mathscr{B}$-free system}.
\section{Tools}\label{section3}
\subsection{Density and F\o lner sequences}
Let $\G$ be a countable Abelian group.
\begin{Def}\label{Banach}
By the \emph{upper Banach density} of $A\subseteq\G$ we mean $$d^*(A)=\sup\{\limsup_{n\to\infty}\frac{|A\cap F_n|}{|F_n|}\ : \ (F_n)_{n\geq1} \text{ is a F\o lner sequence in }\G\}.$$
\end{Def}
We have:
\begin{itemize}
\item $d^*(A)=d^*(A+g)$ for any $A\subseteq\G$ and any $g\in\G$,
\item if $A_i\subseteq\G$ and $d^*(A_i)=1$ for any $i=1,\ldots,n$, then $d^*(\bigcap_{i=1}^nA_i)=1$.
\end{itemize}
\begin{Lemma}\label{cor_Joel}
Let $A\subseteq\G$. Then $d^*(A)=1$ if and only if for any finite set $F\subseteq\G$ there exists $x\in\G$ such that $F+x\subseteq A$. 
\end{Lemma} 
\begin{proof}
Suppose that $d^*(A)=1$. So $d^*(A-g)=1$ for any $g\in\G$. Let $F\subseteq\G$ be finite. Hence also $d^*(\bigcap_{g\in F}(A-g))=1$. Then for any $b\in\bigcap_{g\in F}(A-g)$, we have $F+b\subseteq A$.
In the other direction, if $A$ contains shifts $F_n+x_n$ for some F\o lner sequence $(F_n)_{n\geq1}$ in $\G$ and some $(x_n)_{n\geq1}\subseteq\G$, then these shifts form a new F\o lner sequence, since $|F_n+x_n|=|F_n|$ and $|(F_n+x_n+g)\cap(F_n+x_n)|=|(F_n+g)\cap F_n|$ for any $n\geq1$ and any $g\in\G$. So $d^*(A)=1$.
\end{proof}
\begin{Lemma}\label{Folner_prod}
Let $(F_n)_{n\geq1}$ be a nested F\o lner sequence in $\G$ and $\mathbb{H}\subseteq\G$ be a subgroup with finite index. Then, for sufficiently large $n\geq1$, 
\begin{equation*}
\HH\cap(F_n+x)\neq\emptyset\text{ for any } x\in\G.
\end{equation*}
\end{Lemma}
\begin{proof}
Let $[\G : \HH]=s$ and $g_0=0,g_1,g_2,\ldots,g_{s-1}\in\G$ be such that 
\begin{equation}\label{*}
\G=(\HH+g_0)\cup(\HH+g_1)\cup\ldots\cup(\HH+g_{s-1}).
\end{equation} 
Let $\vep\in(0,\frac{1}{s})$ and $n_0\geq1$ be such that for any $n\geq n_0$, we have
\begin{equation}\label{**}
|(F_n+g_j)\cap F_n|>(1-\vep)|F_n|
\end{equation}
for $j=1,\ldots,{s-1}$. Let $x\in\G$. By \eqref{*}, there exists $0\leq j\leq s-1$ such that $|(F_n+x)\cap(\HH+g_j)|\geq \frac{1}{s}|F_n+x|=\frac{1}{s}|F_n|$. Hence
\begin{equation}\label{***}
|(F_n-g_j+x)\cap\HH|\geq \frac{1}{s}|F_n|.
\end{equation} Notice that 
\begin{align}\label{****}
\begin{split}
|(F_n+x)\cap\HH|&\geq|(F_n+x)\cap(F_n-g_j+x)\cap\HH|\\
&=|(F_n-g_j+x)\cap\HH|-|(F_n+x)^c\cap(F_n-g_j+x)\cap\HH|\\
&\geq|(F_n-g_j+x)\cap\HH|-|(F_n+x)^c\cap(F_n-g_j+x)|\\
&=|(F_n-g_j+x)\cap\HH|-|(F_n^c+x)\cap(F_n-g_j+x)|\\
&=|(F_n-g_j+x)\cap\HH|-|F_n^c\cap(F_n-g_j)|.
\end{split}
\end{align}
Therefore and by \eqref{***} and \eqref{**}, we obtain 
$$|(F_n+x)\cap\HH|\geq\left(\frac{1}{s}-\vep\right)|F_n|>0.$$
In particular, $(F_n+x)\cap\HH\neq\emptyset$. 
\end{proof}
\begin{Lemma}[see {\cite[Chapter I, Proposition 2.2]{MR1878556}}]\label{intersect}
Let $\HH_1,\HH_2\subseteq\G$ be subgroups with finite indices. Then $\HH_1\cap\HH_2$ is also a subgroup with finite index in $\G$ and $[\G : \HH_1\cap\HH_2]=[\G : \HH_1]\cdot[\HH_1 : \HH_1\cap\HH_2]$.
\end{Lemma}
\subsection{Proximality}
Let $(T_g)_{g\in\mathbb{G}}$ be an action of a countable Abelian group $\mathbb{G}$ by homeomorphisms on a compact metric space $(X,D)$.
A pair of points $(x,y)\in X\times X$ is called \emph{proximal} if $$\liminf_{g\to\infty}D(T_gx,T_gy)=0.$$ By $g_k\to\infty$ we mean that for any finite subset $F\subseteq\mathbb{G}$ there exists $K\geq1$ such that for any $k\geq K$ we have $g_k\not\in F$. By $Prox$ we denote the set of all proximal pairs in $X\times X$. A system $(X,(T_g)_{g\in\mathbb{G}})$ is called \emph{proximal} if $Prox=X\times X$, see \cite{MR1958753}.
\begin{Remark}\label{equiv_prox}
A pair of points $(x,y)\in X\times X$ is proximal if and only if there exist a sequence $(g_i)_{i\geq1}\subseteq\G$ and a point $z\in X$ such that $$\lim_{i\to\infty}D(T_{g_i} x,z)=\lim_{i\to\infty}D(T_{g_i}y,z)=0.$$
\end{Remark}
We have the following:
\begin{Lemma}[{see \cite[Prop. 2.2]{AK} for $\G=\Z$}]\label{fixed}
Let $X$ be a compact metric space and $(X,(T_{g})_{g\in\G})$ be a topological dynamical system, where $\G$ is a countable Abelian group. Then a system $(X,(T_{g})_{g\in\G})$ is proximal if and only if it has a fixed point which is the unique minimal subset of $X$.
\end{Lemma}
We skip the proof as it goes the same lines as for $\G=\Z$.

A subset $W\subseteq\G$ is called \emph{syndetic} if there exists a finite set $K\subseteq\G$ such that $W+K=\G$. Let $\vep>0$ and $x,y\in\mathcal{A}^\G$. Put
$$W_{x,y,\vep}:=\{g\in\G\ : \ D(S_gx,S_gy)<\vep\},$$
where $D$ is a metric on $\mathcal{A}^\G$ given by \eqref{metric} for some nested F\o lner sequence $(F_n)_{n\geq1}$ in $\G$. A pair of points $(x,y)\in X\times X$ is called \emph{syndetically proximal} if $W_{x,y,\vep}$ is syndetic for any $\vep>0$. Denote by $SynProx$ the set of all syndetically proximal pairs in $X\times X$. A system $(X,(T_g)_{g\in\G})$ is called \emph{syndetically proximal} if $SynProx=X\times X$. We have $SynProx\subseteq Prox$.
\begin{Th}[{\cite[Theorem 1]{Cl}}]\label{Cl}
The relation $SynProx$ is an equivalence relation.
\end{Th}
\begin{Th}[\cite{Wu}]\label{Wu}
If the relation $Prox$ is transitive, then $Prox=SynProx$.
\end{Th}
As an immediate consequence, we obtain the following:
\begin{Cor}[cf.\ {~\cite[Theorem 19]{MR3205496}} for $\mathbb{G}=\Z$]\label{pr:4}
Let $x_0\in X$ be a fixed point for $(T_g)_{g\in\G}$. Then the following are equivalent:
\begin{itemize}
\item $(X,(T_g)_{g\in\G})$ is syndetically proximal,
\item for any $x\in X$ and any $\vep>0$ the set $W_{x,x_0,\vep}$ is syndetic,
\item $(X,{(T_g)}_{g\in \mathbb{G}})$ is proximal.
\end{itemize}
\end{Cor}
\subsection{Proximality in subshifts}
Assume additionally that $\mathcal{A}=\{0,1\}$ and let $(S_g)_{g\in\G}$ be given by \eqref{eq:fo:2}. Fix a nested F\o lner sequence $(F_n)_{n\geq1}$ in $\G$ and let $D$ be the corresponding metric on $\mathcal{A}^\G$, as in \eqref{metric}. Finally, let ${x_0}:=\mathbf{0}$. 
\begin{Remark}\label{syndetic}
Let $W_{x,F_n}=\{g\in\G\ : \ S_g x|_{F_n}\equiv0\}$ for $x\in\mathcal{A}^\G$ and $n\geq1$. Then for $n=[\log_2\frac{1}{\vep}]+2$, we have $W_{x,F_n}\subseteq W_{x,\mathbf{0},\vep}$. Hence to show that $W_{x,\mathbf{0},\vep}$ is syndetic for any $x\in X$ and any $\vep>0$, we only need to prove that $W_{x,F_n}$ is syndetic for any $x\in X$ and any $n\geq1$.
\end{Remark}
\begin{Lemma}\label{syndetic1}
Let $y\in X$ be a transitive point, i.e., the orbit $\{S_gy\ : \ g\in\G\}$ of $y$ is dense in $X$. Let $n\geq1$. If the set $W_{y,F_n}$ is syndetic, then the set $W_{x,F_n}$ is also syndetic for any $x\in X$.
\end{Lemma}
\begin{proof}
Assume that $W_{y,F_n}$ is syndetic. Then there exists a finite set $K\subseteq\G$ such that $W_{y,F_n}+K=\G$. Without loss of generality, we can assume that $K=-K$. We claim that $W_{x,F_n}+K=\G$. Indeed, let $g\in\G$. Because $(F_n)_{n\geq1}$ is nested and $K$ is finite, there exists $\ell\in\N$ such $K+F_n\subseteq F_\ell$. Since $y$ is transitive, there exists $h\in\G$ such that
\begin{equation}\label{eta}
 x|_{g+F_\ell}= y|_{g+h+F_\ell}.
\end{equation}
By the definition of $K$, there exist $g'\in W_{y,F_n}$ and $g''\in K$ such that
\begin{equation}\label{k}
g+h=g'+g''.
\end{equation}
By \eqref{k}, $-g''+F_n\subseteq F_\ell$, \eqref{eta} and \eqref{k} again, we obtain
$$x|_{g'-h+F_n}=x|_{g-g''+F_n}=y|_{g+h-g''+F_n}=y|_{g'+F_n}=0,$$
so $g'-h\in W_{x,F_n}$. Hence, $g=(g'-h)+g''\in W_{x,F_n}+K$ and the assertion holds.
\end{proof}
Let $\mathfrak{B}$ be a collection of subgroups with finite indices in $\G$, $A\colon\G\to\G$ be an automorphism, $\mathcal{M}_{\mathfrak{B}_A}:=\bigcup_{b\in\mathfrak{B}}A(\mathfrak{b})$, $\mathcal{F}_{\mathfrak{B}_A}:=\G\setminus\mathcal{M}_{\mathfrak{B}_A}$, $\eta_A:=\raz_{\mathcal{F}_{\mathfrak{B}_A}}\in\{0,1\}^\G$ and let $X_{\eta_A}$ be the closure of the set $\{S_g\eta_A\ :\  g\in\G\}$ with respect to the product topology.
\begin{Lemma}\label{d->a}
Let $n\geq1$. If $W_{\eta_A,F_n}$ is non-empty, then $W_{\eta_A,F_n}$ is syndetic.
\end{Lemma}
\begin{proof}
Notice that
\begin{equation*}
W_{\eta_A,F_n}=\{g\in\G : S_g\eta_A|_{F_n}\equiv0\}=\{g\in\G : \eta_A|_{F_n+g}\equiv0\}=\{g\in\G : F_n+g\subseteq\mathcal{M}_{\mathfrak{B}_A}\}.
\end{equation*}
Take $g\in W_{\eta_A,F_n}$, let $s=|F_n|$ and let $\mathfrak{b}_1,\ldots,\mathfrak{b}_s\in\mathfrak{B}$ be such that 
$F_n+g\subseteq\bigcup_{i=1}^sA(\mathfrak{b}_i)$.
Let $\mathbb{H}:=\bigcap_{1\leq i\leq s} A(\mathfrak{b}_i)$, cf. Lemma \ref{intersect}. Then $F_n+g+\mathbb{H}\subseteq\bigcup_{i=1}^sA(\mathfrak{b}_i)\subseteq\mathcal{M}_{\mathfrak{B}_A}$. In other words, $\eta_A|_{g+\mathbb{H}+F_n}=0$ and it follows that 
\begin{equation}\label{par}
g+\mathbb{H}\subseteq W_{\eta_A,F_n}.
\end{equation}
Since $\mathbb{H}$ has finite index, $g+\mathbb{H}$ is syndetic. Hence, the assertion follows by \eqref{par}.
\end{proof}
\begin{Remark}\label{automorphism}
Notice that $(X_{\eta_A},(S_g)_{g\in\G})$ is proximal if and only if $(X_{\eta},(S_g)_{g\in\G})$ is proximal. Indeed, since for any $\mathfrak{b}\in\mathfrak{B}$ we have $h+g\in A(\mathfrak{b})$ if and only if $A^{-1}(h+g)\in\mathfrak{b}$, so $$W_{\eta_A,F_n}=\{g\in\G \ : \ S_{A^{-1}g}\eta|_{A^{-1}F_n}\equiv0\}=\{Ag\in\G \ : \ S_g\eta|_{A^{-1}F_n}\equiv0\}=A(W_{\eta,A^{-1}F_n}).$$ To conclude, we use the fact that automorphisms send syndetic sets into syndetic sets.
\end{Remark}
\subsection{Ideals in number fields}
Recall that proper subgroups $\mathbb{H}_1,\mathbb{H}_2\subseteq\G$ are said to be \emph{coprime} whenever $\mathbb{H}_1+\mathbb{H}_2=\G$.
\begin{Th}[Chinese Remainder Theorem, see e.g. {\cite[Chapter II, Theorem 2.1]{MR1878556}}]
Let $R$ be a commutative ring, and let $I_1,\ldots, I_n$ be pairwise coprime ideals in $R$. If $a_1, \ldots, a_n$ are elements of $R$, then there exists $a\in R$ such that
$a\equiv a_i \bmod I_i$, $i=1, \ldots , n$.
\end{Th}
\begin{Lemma}[Prime Avoidance Lemma, see e.g. {\cite[Proposition 1.11]{MR0242802}}]\label{avoid}
Let $R$ be a commutative ring with unity. Let $\mathfrak{p}_1,\ldots,\mathfrak{p}_s$ be prime ideals and $\mathfrak{a}$ be an ideal in $R$. If $\mathfrak{a}\subseteq\bigcup_{i=1}^s\mathfrak{p}_i$ then $\mathfrak{a}\subseteq\mathfrak{p}_{j_0}$ for some $1\leq j_0\leq s$.
\end{Lemma}
\begin{Lemma}[{\cite[2.4. p. 128]{Hall}}]\label{index}
If $\G$ is finitely generated then $\G$ contains only finite number (possibly zero) of subgroups of a given finite index $n$.
\end{Lemma}
\subsection{Lattices}
Our main result in this section is the following:
\begin{Prop}\label{prime}
Let $\{\Lambda_i\}_{i\geq1}$ be a pairwise coprime family of proper lattices. Then $\left\{[\Z^m:\Lambda_i]\right\}_{i\geq1}$ contains an infinite pairwise coprime set.
\end{Prop}
\begin{Remark}\label{square}
Recall that for $\Lambda=\sum_{j=1}^m(a_{1,j},\ldots,a_{m,j})\Z$, we have $[\Z^m:\Lambda]=|\det(a_{i,j})_{1\leq i,j\leq m}|$ and $\Lambda$ is proper if and only if $[\Z^m:\Lambda]\geq2$. Moreover, there exist $\widetilde{a_{i,j}}\in\Z$, $1\leq j\leq i\leq m$, with
\begin{equation}
\Lambda=(\widetilde{a_{1,1}},\ldots,\widetilde{a_{m,1}})\Z+(0,\widetilde{a_{2,2}},\ldots,\widetilde{a_{m,2}})\Z+\ldots+(0,\ldots,0,\widetilde{a_{m,m}})\Z
\end{equation}
 (see, e.g. \cite[Theorem II.1]{MR0340283}). Then 
\begin{equation*}
[\Z^m:\Lambda]\Z^m=\left(\det(a_{i,j})_{1\leq i,j\leq m}\right)\Z^m=\left(\prod_{j=1}^m\widetilde{a_{j,j}}\right)\Z^m\subseteq\Lambda
\end{equation*} (see, e.g. \cite[Lemma 2.11]{MR3607789}). 
\end{Remark}
\begin{proof}[Proof of Proposition \ref{prime}]
Suppose that $\{[\Z^m:\Lambda_i]\}_{i\geq1}$ does not contain an infinite
pairwise coprime subset. By Theorem \ref{one-dim}, there exist $q_1,\ldots,q_n>1$, such that $\{[\Z^m:\Lambda_i]\}_{i\geq1}\subseteq\bigcup_{i=1}^nq_i\Z$. We can assume without loss of generality that $q_1,\ldots,q_n$ are primes. In view of Remark \ref{square}, we may also assume that 
\begin{equation}\label{form}
\Lambda_i=(a_{1,1}^{(i)},\ldots,a_{m,1}^{(i)})\Z+(0,a_{2,2}^{(i)},\ldots,a_{m,2}^{(i)})\Z+\ldots+(0,\ldots,0,a_{m,m}^{(i)})\Z
\end{equation} for $i\geq1$. Then, there exist $p\in\{q_1,\ldots,q_n\}$ and $0\leq r_{k,j}<p$, $1\leq k,j\leq m$ such that $p\divides\prod_{j=1}^ma_{j,j}^{(i)}$ and \begin{equation}\label{eq1}
a_{k,j}^{(i)} \equiv r_{k,j}\bmod p\text{ for any }1\leq k,j\leq m
\end{equation} 
for infinitely many $i\geq1$.
Since $p\divides\prod_{j=1}^ma_{j,j}^{(i)}$, there exists $1\leq j_0\leq m$ such that $p\divides a_{j_0,j_0}^{(i)}$ and \eqref{eq1} holds for infinitely many $i\geq1$. Notice that there is at most one $i\geq1$ such that $p\divides a_{1,1}^{(i)}$, otherwise if $p\divides a_{1,1}^{(i')}$ then $\Lambda_i$ and $\Lambda_{i'}$ are not coprime -- we cannot get $(1,0,\ldots,0)\in\Lambda_i+\Lambda_{i'}$. So $1<j_0\leq m$. Let $j_0$ be the smallest integer such that $p\divides a_{j_0,j_0}^{(i)}$ and $p\ndivides a_{j,j}^{(i)}$, $1\leq j<j_0$, for infinitely many $i\geq1$ such that \eqref{eq1} holds. Since $\Lambda_i$, $\Lambda_{i'}$ are coprime, for any $1\leq t< j_0$ there exist $\ell_1^{(t)},\ldots,\ell_m^{(t)},s_1^{(t)},\ldots,s_m^{(t)}\in\Z$ such that 
\begin{equation}\label{eq2}
\begin{alignedat}{5}
&\ell_1^{(t)}(a_{1,1}^{(i)},\ldots,a_{m,1}^{(i)})&&+\ell_2^{(t)}(0,a_{2,2}^{(i)},\ldots,a_{m,2}^{(i)})&+\ldots&+\ell_m^{(t)}(0,\ldots,0,a_{m,m}^{(i)})\\
+&s_1^{(t)}(a_{1,1}^{(i')},\ldots,a_{m,1}^{(i')})&&+s_2^{(t)}(0,a_{2,2}^{(i')},\ldots,a_{m,2}^{(i')})&+\ldots&+s_m^{(t)}(0,\ldots,0,a_{m,m}^{(i')})\\
=&(\underbrace{0,\ldots, 0}_{j_0-t-1},\underbrace{1,\ldots, 1}_{t},\underbrace{0,\ldots,0}_{m+1-j_0}).\hspace{-4em}&&&
\end{alignedat}
\end{equation}
By \eqref{eq1} and \eqref{eq2} for $t=1$, we get
\begin{equation}\label{eq4}
\begin{alignedat}{5}
&(\ell_1^{(1)}+s_1^{(1)})r_{1,1}&& & & &\equiv0\bmod p, \\
&(\ell_1^{(1)}+s_1^{(1)})r_{2,1}&&+(\ell_2^{(1)}+s_2^{(1)})r_{2,2}  & & & \equiv0\bmod p,\\
&\phantom{allaa}\vdots &&&&\\
&(\ell_1^{(1)}+s_1^{(1)})r_{j_0-1,1}&&+(\ell_2^{(1)}+s_2^{(1)})r_{j_0-1,2} &+\ldots &+(\ell_{j_0-1}^{(1)}+s_{j_0-1}^{(1)})r_{j_0-1,j_0-1} & \equiv1\bmod p, \\
&(\ell_1^{(1)}+s_1^{(1)})r_{j_0,1}&&+(\ell_2^{(1)}+s_2^{(1)})r_{j_0,2}&+\ldots&+(\ell_{j_0}^{(1)}+s_{j_0}^{(1)})r_{j_0,j_0}&\equiv0\bmod p.
\end{alignedat}
\end{equation}
Since $p\ndivides a_{j,j}^{(i)}$ for $1\leq j<j_0$, by \eqref{eq1} and \eqref{eq4}, we get $p\ndivides r_{j,j}$ for any $1\leq j<j_0$ and $p\divides(\ell_j^{(1)}+s_j^{(1)})$ for any $1\leq j\leq j_0-2$ and $p\ndivides(\ell_{j_0-1}^{(1)}+s_{j_0-1}^{(1)})$. Moreover, by the bottom line in \eqref{eq4}, we have $p\divides(\ell_{j_0-1}^{(1)}+s_{j_0-1}^{(1)})r_{j_0,j_0-1}$, so $p\divides r_{j_0,j_0-1}$. By similar reasoning for $t=2,\ldots,j_0-1$, we can show that $p\divides r_{j_0,j_0-t}$. But then $p\divides a_{j_0,t}^{(i)},a_{j_0,t}^{(i')}$ for any $i,i'$ and for any $t=1,\ldots,j_0$. So $\Lambda_i+\Lambda_{i'}\subseteq\Z^{j_0-1}\times p\Z\times\Z^{m-j_0}$. This contradicts that elements of $\{\Lambda_i\}_{i\geq1}$ are pairwise coprime.
\end{proof}
\section{Proximality of $(X_\eta,(S_a)_{a\in\mathcal{O}_K^m})$}
\begin{proof}[Proof of Theorem \ref{ideals_general}]
\eqref{giC} $\implies$ \eqref{giD}.
Since $(X_\eta,(S_a)_{a\in\OK^m})$ is proximal, by Lemma \ref{fixed}, the system $(X_\eta,(S_a)_{a\in\OK^m})$ has a unique fixed point, i.e., $\mathbf{0}\in X_\eta$ or $\mathbf{1}\in X_\eta$. Suppose that $\mathbf{1}\in X_\eta$. Then, for any nested F\o lner sequence $(F_n)_{n\geq1}$ in $\OK^m$, there exists $(x_n)_{n\geq1}\subseteq\OK^m$ such that 
\begin{equation}\label{Fn}
F_n+x_n\subseteq\mathcal{F}_\mathfrak{B}.
\end{equation}
However, by Lemma \ref{Folner_prod}, for $n\geq1$ sufficiently large, we have $\mathfrak{b}\cap(F_n+x_n)\neq\emptyset$ for any $\mathfrak{b}\in\mathfrak{B}$, which contradicts \eqref{Fn}. It follows that $\mathbf{0}\in X_\eta$.

\eqref{giD} $\implies$ \eqref{giF}.
Suppose that \eqref{giD} holds and \eqref{giF} does not hold. Then, for some $k\geq1$, \begin{equation}\label{b->c1}
\mathcal{M}_\mathfrak{B}\subseteq\bigcup_{j=1}^k I_j
\end{equation} for some proper ideals $I_1,\ldots,I_k$ with finite indices in $\OK^m$. Let $(F_n)_{n\geq1}$ be a nested F\o lner sequence in $\OK^m$. By \eqref{giD}, for any $n\geq1$, there exists $a_n\in\OK^m$ such that \begin{equation}\label{b->c2}
 a_n+F_n\subseteq\mathcal{M}_\mathfrak{B}\subseteq\bigcup_{j=1}^k I_j.
\end{equation} By Lemma \ref{intersect}, $I:=\bigcap_{j=1}^k I_j$ is an ideal with finite index in $\OK^m$. Hence, we have
$\OK^m=\sqcup_{\ell=1}^L\left(I+c_\ell\right)$ for some $c_1,\ldots, c_L\in\OK^m$ and \begin{equation}\label{b->c3}
\OK^m\setminus\bigcup_{j=1}^k I_j=\sqcup_{c_\ell\not\in\bigcup_{j=1}^k I_j}\left(I+c_\ell\right).
\end{equation} Let $1\leq \ell\leq L$ be such that $c_\ell\not\in\bigcup_{j=1}^k I_j$. By Lemma \ref{Folner_prod}, for sufficiently large $n\geq1$, we have $I\cap (F_n+x)\neq\emptyset$ for any $x\in\OK^m$. In particular, by taking $x=a_n-c_\ell$, we obtain $(I+c_\ell)\cap(F_n+a_n)\neq\emptyset$. But by \eqref{b->c2} and \eqref{b->c3}, we have $(F_n+a_n)\cap(I+c_\ell)=\emptyset$. This is a contradiction.

\eqref{giF} $\implies$ \eqref{giG}.
We will proceed inductively. Fix $\mathfrak{c}_1\in \mathfrak{B}$. Suppose that for $k\ge 1$ we have found pairwise coprime subset $\{\mathfrak{c}_1,\ldots,\mathfrak{c}_k\}\subseteq \mathfrak{B}$. Then, we have $\mathfrak{c}_i=I_1^{(i)}\times\ldots\times I_m^{(i)}$ for some non-zero ideals $I_1^{(i)},\ldots, I_m^{(i)}\subseteq\OK$, $1\leq i\leq k$. Since $\OK$ is a Dedekind domain, the number of prime ideals $\mathfrak{q}$ such that $I_\ell^{(i)}\subseteq\mathfrak{q}$ for some $1\leq i\leq k$ and some $1\leq\ell\leq m$ is finite. Let $\{\mathfrak{q}_j\}_{j=1}^t$ be all such $\mathfrak{q}$. Then $\OK^{s-1}\times\mathfrak{q}_j\times\OK^{m-s}$ is a prime ideal in $\OK^m$ for any $1\leq j\leq t$ and any $1\leq s\leq m$. By \eqref{giF}, there exists $b\in\OK^m$ such that $b\in\mathcal{M}_\mathfrak{B}\setminus\bigcup_{j=1}^t\bigcup_{s=1}^m\OK^{s-1}\times\mathfrak{q}_j\times\OK^{m-s}$. Let $\mathfrak{c}_{k+1}\in\mathfrak{B}$ be such that $b\in\mathfrak{c}_{k+1}$. Then 
\begin{equation}\label{c}
\mathfrak{c}_{k+1}\not\subseteq\bigcup_{j=1}^t\bigcup_{s=1}^m\OK^{s-1}\times\mathfrak{q}_j\times\OK^{m-s}.
\end{equation}
We have $\mathfrak{c}_{k+1}=I_1^{(k+1)}\times\ldots\times I_m^{(k+1)}$ for some non-zero ideals $I_1^{(k+1)},\ldots,I_m^{(k+1)}\subseteq\OK$. By \eqref{c}, we have $I_\ell^{(k+1)}\not\subseteq\mathfrak{q}_j$ for any $1\leq\ell\leq m$ and any $1\leq j\leq t$. So $\mathfrak{c}_{k+1}$ is coprime with each of $\OK^{s-1}\times\mathfrak{q}_j\times\OK^{m-s}$, $j=1,\ldots,t$, $s=1,\ldots,m$. Hence $\mathfrak{c}_{k+1}$ is also coprime with each of $\mathfrak{c}_1,\ldots,\mathfrak{c}_k$.

\eqref{giG} $\implies$ \eqref{giC}. 
By Theorem \ref{pr:4}, Remark \ref{syndetic} and Lemma \ref{syndetic1}, we need to show
$$W_{\eta,F_n}=\{a\in\OK^m\ :\ S_{a} \eta|_{F_n}\equiv0\}=\{a\in\OK^m\ :\ \eta|_{a+F_n}\equiv0\}=\{a\in\OK^m\ : \ a+F_n\subseteq\mathcal{M}_\mathfrak{B}\}$$
is syndetic for any $n\in\N$. Let $\{I_i\}_{i\geq1}\subseteq\mathfrak{B}$ be infinite pairwise coprime and $F_n:=\{f_1,\ldots, f_s\}$, where $|F_n|=s$. By the Chinese Remainder Theorem applied to $I_1,\ldots,I_s$ and $-f_1,\ldots,-f_s$, there exists $a\in\OK^m$ such that $a\equiv-f_i\ \bmod I_i$ for any $1\leq i\leq s$. It follows that $a+F_n\subseteq\mathcal{M}_\mathfrak{B}$. Therefore, $W_{\eta,F_n}\neq\emptyset$. In view of Lemma \ref{d->a}, it follows that $W_{\eta,F_n}$ is syndetic.

\eqref{giD} $\implies$ \eqref{giH}
Suppose that \eqref{giD} holds and \eqref{giH} does not hold. Then there exist $a\in\OK^m$ and an ideal $I\subseteq\OK^m$ with finite index such that $I+a\subseteq\mathcal{F}_\mathfrak{B}$. Let $(F_n)_{n\geq1}$ be a nested F\o lner sequence in $\OK^m$. By \eqref{giD}, for any $n\geq1$ there exists $a_n\in\OK^m$ such that 
\begin{equation}\label{sub}
 F_n+a_n\subseteq\mathcal{M}_\mathfrak{B}.
\end{equation} By Lemma \ref{Folner_prod}, for sufficiently large $n\geq1$, we have $(F_n-a+a_n)\cap I\neq\emptyset$. Since $I+a\subseteq\mathcal{F}_{\mathfrak{B}}$, this contradicts \eqref{sub}.

\eqref{giH} $\implies$ \eqref{giF}.
Suppose that \eqref{giF} does not hold and let $I_1,\dots,I_k\subseteq\OK^m$ be proper ideals with finite indices such that 
\begin{equation}\label{I}
\mathcal{M}_\mathfrak{B}\subseteq\bigcup_{j=1}^k I_j.
\end{equation} Let $M:=\bigcap_{j=1}^k I_j$ and $\underline{1}:=(\underbrace{1,\ldots,1}_m)$. We claim that $M+\underline{1}\subseteq \mathcal{F}_\mathfrak{B}$.
Indeed, suppose that $(M+\underline{1})\cap\mathcal{M}_\mathfrak{B}\neq\emptyset$. Then, there are $a\in M$ and $i\geq1$ such that $a+\underline{1}\in\mathfrak{b}_i$. By \eqref{I}, there is $1\leq j\leq k$ such that $a+\underline{1}\in I_j$. Since $a\in M\subseteq I_j$, it follows $\underline{1}=(a+\underline{1})-a\in I_j$. Therefore, $I_j=\OK^m$, which contradicts the choice of $I_1,\ldots,I_k$.

\eqref{giD}$\iff$\eqref{gdens}
Notice that $\mathbf{0}\in X_\eta$ if and only if for any finite $F\subseteq\OK^m$ there exists $x\in\OK^m$ such that $\eta_{|F+x}\equiv0$. The assertion follows from the definition of $\eta$ and Lemma \ref{cor_Joel}.
\end{proof}
\section{Proximality of $(X_\eta,(S_\mathbf{n})_{\mathbf{n}\in\Z^m})$}
\begin{Remark}\label{product}
Let $\{\Lambda_i\}_{i\geq1}$ be as in Corollary \ref{rectangular}, then  $\{\Lambda_i\}_{i\geq1}$ contains an infinite pairwise coprime subset $\{\Lambda_{i_k}\}_{k\geq1}$ precisely if $(a_j^{(i_k)})_{k\geq1}$ are pairwise coprime and for some $1\leq j_0\leq m$, $\{a_{j_0}^{(i_k)}\}_{k\geq1}$ is infinite. Indeed, if all $\{a_j^{(i_k)}\}_{k\geq1}$ were finite then $\{\Lambda_{i_k}\}_{k\geq1}$ would be finite, too.
\end{Remark}
\begin{Remark}\label{no-imply}
Given $\eta_j=\raz_{\Z\setminus\bigcup_{i\geq1}a_j^{(i)}\Z}$, $j=1,2$ and the corresponding orbit closures $X_{\eta_j}\subseteq\{0,1\}^\Z$, it is natural to study the product $\Z^2$-action $(X_{\eta_1}\times X_{\eta_2},(\widetilde{S}_{\mathbf{n}})_{\mathbf{n}\in\Z^2})$, where $\widetilde{S}_{(n,m)}(x,y)=(S^nx,S^my)$, $n,m\in\Z$. By the definition of the product $\Z^2$-action and by Theorem \ref{one-dim}, the following conditions are equivalent:
\begin{itemize}
\item $(X_{\eta_1}\times X_{\eta_2},(\widetilde{S}_{\mathbf{n}})_{\mathbf{n}\in\Z^2})$ is proximal,
\item $(X_{\eta_j},S)$ for $j=1,2$ is proximal,
\item for $j=1,2$ we have $\{a_j^{(i)}\}_{i\geq1}$ contains an infinite pairwise coprime subset.
\end{itemize}
On the other hand, in view of Remark \ref{product}, for $\eta=\raz_{\Z^2\setminus\bigcup_{i\geq1}a_1^{(i)}\Z\times a_2^{(i)}\Z}$, the following are equivalent:
\begin{itemize}
\item $(X_\eta,(S_\n)_{\n\in\Z^2})$ is proximal,
\item there exists $(i_k)_{k\geq1}$ such that $\{a_j^{(i_k)}\}_{k\geq1}$ is pairwise coprime for $j=1,2$ with $\{a_{j_0}^{(i_k)}\}_{k\geq1}$ infinite for some $j_0\in\{1,2\}$. 
\end{itemize}
Clearly, this shows that the proximality of the two $\Z^2$-actions $(X_{\eta_1}\times X_{\eta_2},(\widetilde{S}_{\mathbf{n}})_{\mathbf{n}\in\Z^2})$ and $(X_\eta,(S_\n)_{\n\in\Z^2})$ are independent of one another (there is no implication in either direction). 
\end{Remark}
\begin{Remark}
Recently, Baake, Huck and Strungaru \cite{MB-CH-NS} studied the maximal density of weak model sets given by a pairwise coprime family of sublattices $\{\Lambda_i\}_{i\geq1}$ of a lattice $\Lambda\subset\R^m$ (a subgroup of the additive group $\R^m$ which is isomorphic to the additive group $\Z^m$) with the (absolute) convergence condition $\sum_{i\geq1}\frac{1}{[\Lambda:\Lambda_i]}<\infty$ and $\Lambda_F+\Lambda_{F'}=\Lambda_{F\cap F'}$ for all finite $F,F'\subseteq\N$, where $\Lambda_F:=\bigcap_{n\in F}\Lambda_n$ and $\Lambda_\emptyset=\Lambda$. Notice that the third condition, i.e., $\Lambda_F+\Lambda_{F'}=\Lambda_{F\cap F'}$ for all finite $F,F'\subseteq\N$, holds for lattices that are ideals. In particular, the characterization of maximal natural density of the weak model set from~\cite{MB-CH-NS} can be applied to pairwise coprime lattices in the same form as in Corollary \ref{rectangular} and satisfying the convergence condition.
\end{Remark}
Let us now state the ,,lattice analogue'' of Theorem \ref{ideals_general}:
\begin{Th}\label{general_lattices} Suppose that $\Lambda_i$ is a lattice in $\Z^m$, $i\geq1$. Consider the following conditions:
\begin{enumerate}[(a)]
\item $(X_\eta,(S_{\mathbf{n}})_{\mathbf{n}\in\Z^m})$ is proximal,\label{la}
\item $\mathbf{0}\in X_\eta$, where $\mathbf{0}_\mathbf{n}=0$ for any $\mathbf{n}\in\Z^m$,\label{lb}
\item for any lattices $\widetilde{\Lambda}_1,\widetilde{\Lambda}_2,\ldots, \widetilde{\Lambda}_k$ such that $\bigcup_{j=1}^k\widetilde{\Lambda}_j\neq\Z^m$, we have $\mathcal{M}_\mathscr{B}\not\subseteq\bigcup_{j=1}^k \widetilde{\Lambda}_j$,\label{lc}
\item $\{\Lambda_i\}_{i\geq1}$ contains an infinite pairwise coprime subset,\label{ld}
\item for any $\mathbf{n}\in\Z^m$ and any lattice $\Lambda\subseteq\Z^m$ we have $\mathbf{n}+\Lambda\not\subseteq\mathcal{F}_\mathscr{B}$,\label{le}
\item $d^*(\mathcal{M}_\mathscr{B})=1$.\label{lf}
\end{enumerate}  
Then: \eqref{ld}$\implies$\eqref{la}$\iff$\eqref{lb}$\iff$\eqref{lc}$\iff$\eqref{le}$\iff$\eqref{lf}.
\end{Th}
\begin{proof}[Proof of Theorem \ref{general_lattices}]
The proof of \eqref{la} $\implies$ \eqref{lb} $\implies$ \eqref{lc} and \eqref{lb} $\iff$ \eqref{lf} goes along the same lines as in Theorem \ref{ideals_general} as they use tools from Section \ref{section3} valid in countable Abelian groups.

\eqref{lc} $\implies$ \eqref{le}
Suppose that $\mathbf{n}+\Lambda\subseteq\mathcal{F}_\mathscr{B}$ and consider $\Lambda'_i:=\Lambda_i+\Lambda$. We claim that $\bigcup_{i\geq1}\Lambda'_i=\bigcup_{i\in C}\Lambda'_i$, where $C$ is finite, and $\mathcal{M}_\mathscr{B}\subseteq\bigcup_{i\geq1}\Lambda'_i$ with $\bigcup_{i\geq1}\Lambda'_i\neq\Z^m$ which will contradict \eqref{lc}. Indeed, it follows by Lemma \ref{intersect}, that $[\Z^m:\Lambda]=[\Z^m:\Lambda_i']\cdot[\Lambda_i':\Lambda]$, i.e., $[\Z^m:\Lambda'_i]\mid[\Z^m:\Lambda]$. Therefore, using Lemma \ref{index}, $|\{\Lambda'_i\ ; i\geq1\}|=:|C|<\infty$. Let $\{\Lambda'_i\ ; i\geq1\}=\{\widetilde{\Lambda_i} \ ; i\in C\}$. We claim that $\mathbf{n}\not\in\bigcup_{i\geq1}\Lambda'_i$. Indeed, if $\mathbf{n}\in\Lambda_i'$ for some $i\geq1$, then $\mathbf{n}=\lambda_i+\lambda$, where $\lambda_i\in\Lambda_i$, $\lambda\in\Lambda$. This yields $\lambda_i=\mathbf{n}-\lambda\in \mathbf{n}+\Lambda\subseteq\mathcal{F}_{\mathscr{B}}$ and on the other hand $\lambda_i\in\mathcal{M}_\mathscr{B}$, which is impossible. In particular, $\bigcup_{i\in C}\Lambda'_i\neq\Z^m$, which completes the proof.

\eqref{le} $\implies$ \eqref{lb}
Suppose that \eqref{lb} does not hold. Let $(F_n)_{n\in\N}$ be a nested F\o lner sequence in $\Z^m$ such that $F_1=\{(\underbrace{0,\ldots,0}_m)\}$. Then clearly $F_1\subseteq\mathcal{M}_\mathscr{B}$. Let $s\in\N$ be the largest integer such that for some $\mathbf{n}\in\Z^m$, 
\begin{equation}\label{n}
\mathbf{n}+F_s\subseteq\mathcal{M}_\mathscr{B}
\end{equation}
(such an $s\in\N$ exists since otherwise we would have $\mathbf{0}\in X_\eta$). Fix $\mathbf{n}\in\Z^m$ such that \eqref{n} holds. For each $t\in\mathbf{n}+F_s$ choose $i_t\geq1$ with $t\in\Lambda_{i_t}$ and set $\Lambda:=\bigcap_{t\in\mathbf{n}+F_s}\Lambda_{i_t}$. Then $\bigcup_{t\in\mathbf{n}+F_s}\Lambda_{i_t}+\bigcap_{t\in\mathbf{n}+F_s}\Lambda_{i_t}\subseteq\bigcup_{t\in\mathbf{n}+F_s}\left(\Lambda_{i_t}+\bigcap_{t'\in\mathbf{n}+F_s}\Lambda_{i_{t'}}\right)\subseteq\bigcup_{t\in\mathbf{n}+F_s}\Lambda_{i_t}$. So  
\begin{equation}\label{(2)}
\mathbf{n}+F_s+\Lambda\subseteq\mathcal{M}_\mathscr{B}.
\end{equation}
By the definition of $s$, we have $\left(F_{s+1}+\mathbf{n}\right)\cap\mathcal{M}_\sB\neq\emptyset$. Let $\mathbf{m}\in\Lambda$ be such that the cardinality of $((F_{s+1}\setminus F_s)+\mathbf{n}+\mathbf{m})\cap\mathcal{M}_\mathscr{B}$ is maximal. For each $u\in((F_{s+1}\setminus F_s)+\mathbf{n}+\mathbf{m})\cap\mathcal{M}_\mathscr{B}$ choose $j_u\geq1$ with $u\in\Lambda_{j_u}$ and set $\Lambda':=\Lambda\cap\bigcap_{u\in((F_{s+1}\setminus F_s)+\mathbf{n}+\mathbf{m})\cap\mathcal{M}_\mathscr{B}}\Lambda_{j_u}$. Similarly as \eqref{(2)}, we show 
\begin{equation}\label{(3')}
\left(\left((F_{s+1}\setminus F_s)+\mathbf{n}+\mathbf{m}\right)\cap\mathcal{M}_\mathscr{B}\right)+\Lambda'\subseteq\mathcal{M}_\mathscr{B}.
\end{equation}
 Since $F_{s+1}=(F_{s+1}\setminus F_s)\cup F_s$ and $\Lambda'\subseteq\Lambda$, it follows by \eqref{(2)} and \eqref{(3')} that \begin{equation}\label{(3)}
 \left((F_{s+1}+\mathbf{n}+\mathbf{m})\cap\mathcal{M}_\mathscr{B}\right)+\Lambda'\subseteq\mathcal{M}_\mathscr{B}
\end{equation}
We claim that in fact for $v\in F_{s+1}+\mathbf{n}+\mathbf{m}$ and $\lambda'\in\Lambda'$,
\begin{equation}\label{(4)}
v\in\mathcal{M}_\mathscr{B} \iff v+\lambda'\in\mathcal{M}_\mathscr{B}. 
\end{equation}
If this is not true then for some $v\in F_{s+1}+\mathbf{n}+\mathbf{m}$ and $\lambda'\in\Lambda'$ we have $v+\lambda'\in\mathcal{M}_\mathscr{B}$ and $v\not\in\mathcal{M}_\mathscr{B}$. Since $\mathbf{m}\in\Lambda'\subseteq\Lambda$, by \eqref{(2)}, we have $v\not\in F_s+\mathbf{n}+\mathbf{m}$. Therefore, in view of \eqref{(3)} and \eqref{(4)}, the cardinality of $\left((F_{s+1}\setminus F_s)+\mathbf{n}+\mathbf{m}+\lambda'\right)\cap\mathcal{M}_\mathscr{B}$ is larger than the cardinality of $\left((F_{s+1}\setminus F_s)+\mathbf{n}+\mathbf{m}\right)\cap\mathcal{M}_\mathscr{B}$ which contradicts the choice of $\mathbf{m}$ (as $\mathbf{m}+\lambda'\in\Lambda+\Lambda'\subseteq\Lambda$). Now, it suffices to take $v\in (F_{s+1}+\mathbf{m}+\mathbf{n})\cap\mathcal{F}_\mathscr{B}$ (this set is non-empty by the choice of $s$) and use \eqref{(4)} to see that $v+\Lambda'\subseteq\mathcal{F}_\mathscr{B}$.

\eqref{lb} $\implies$ \eqref{la}
Let $(F_n)_{n\geq1}$ be a nested F\o lner sequence in $\Z^m$.
Similarly as in the proof of \eqref{giG}$\implies$\eqref{giC} in Theorem \ref{ideals_general}, it is enough to show that $W_{\eta,F_n}=\{\mathbf{n}\in\Z^m\ :\ \eta|_{\mathbf{n}+F_n}\equiv0\}$ is non-empty. This follows directly from $\mathbf{0}\in X_\eta$. 

\eqref{ld} $\implies$ \eqref{lb}
Let $(i_k)_{k\geq1}$ be such that $\{\Lambda_{i_k}\}_{k\geq1}$ is infinite and pairwise coprime. By Proposition \ref{prime}, passing to a subsequence if necessary, we may assume that $\{[\Z^m:\Lambda_{i_k}]\}_{k\geq1}$ is pairwise coprime. Clearly, $d_{i_k}\Z^m$, where $d_{i_k}=[\Z^m:\Lambda_{i_k}]$, is an ideal in $\Z^m$ (considered with coordinatewise multiplication). Since $\{d_{i_k}\}_{k\geq1}$ are pairwise coprime, we can apply the Chinese Remainder Theorem to each choice of $a_1, a_2, \ldots, a_k\in\Z^m$ and ideals $d_{i_1}\Z^m,\ldots,d_{i_k}\Z^m$. Then there exists $x\in\Z^m$ with $x-a_\ell\in d_{i_\ell}\Z^m$ for any $1\leq \ell \leq k$. By Remark \ref{square}, we get $d_{i_\ell}\Z^m\subseteq\Lambda_{i_\ell}$. Hence $x-a_\ell\in d_{i_\ell}\Z^m\subseteq\Lambda_{i_\ell}\subseteq\mathcal{M}_\mathscr{B}$. So, we have that $\eta_{x-a_\ell}=0$ for any $1\leq\ell\leq k$, which gives $\mathbf{0}\in X_\eta$.
\end{proof}
\section{Examples}
\begin{Remark}
The assumption $\bigcup_{j=1}^k\widetilde{\Lambda}_j\neq\Z^m$ is necessary in \eqref{lc} above -- we can have $\bigcup_{j=1}^k\widetilde{\Lambda}_j=\Z^m$ with $\widetilde{\Lambda}_j\neq\Z^m$, $1\leq j\leq k$ already for $m=2$: 
$$(\Z\times2\Z)\cup(2\Z\times\Z)\cup((1,1)\Z+(0,2)\Z)=\Z^2.$$
On the other hand, suppose that for proper ideals $I_j\subseteq\OK^m$, $1\leq j\leq k$, with finite indices, we have $\bigcup_{j=1}^k I_j=\OK^m$. For $1\leq j\leq k$, choose a prime ideal $\mathfrak{p}_j\supseteq I_j$. Then by Lemma~\ref{avoid}, $\bigcup_{j=1}^k\mathfrak{p}_j=\OK^m\subseteq \mathfrak{p}_{j_0}\nsubseteq\OK^m$, which is impossible.
\end{Remark}
\begin{Example}[\eqref{lb} $\centernot\implies$ \eqref{ld} in Theorem \ref{general_lattices}]\label{ex2}
Let $\mathscr{B}=\{\Lambda_i\}_{i\geq1}$, where
$\Lambda_1=2\mathbb{Z}\times\mathbb{Z}$, $\Lambda_2=\mathbb{Z}\times2\mathbb{Z}$, $\Lambda_i=(1,1)\mathbb{Z}+(0,2p_i)\mathbb{Z}$, $i\geq3$ and $\{p_i\}_{i\geq3}$ is the set of all primes. Clearly, $\mathcal{F}_\mathscr{B}\subseteq(2\Z+1)\times(2\Z+1)$. Moreover, $(n,m)\in\left((2\Z+1)\times(2\Z+1)\right)\cap\mathcal{M}_\mathscr{B}$ precesily when $2\ndivides n$ and $2p_i\divides m-n$ for some $i\geq3$. Equivalently, $2\ndivides n,m$ and $m-n\neq\pm2$. Thus, \begin{center}
\includegraphics[scale=0.5]{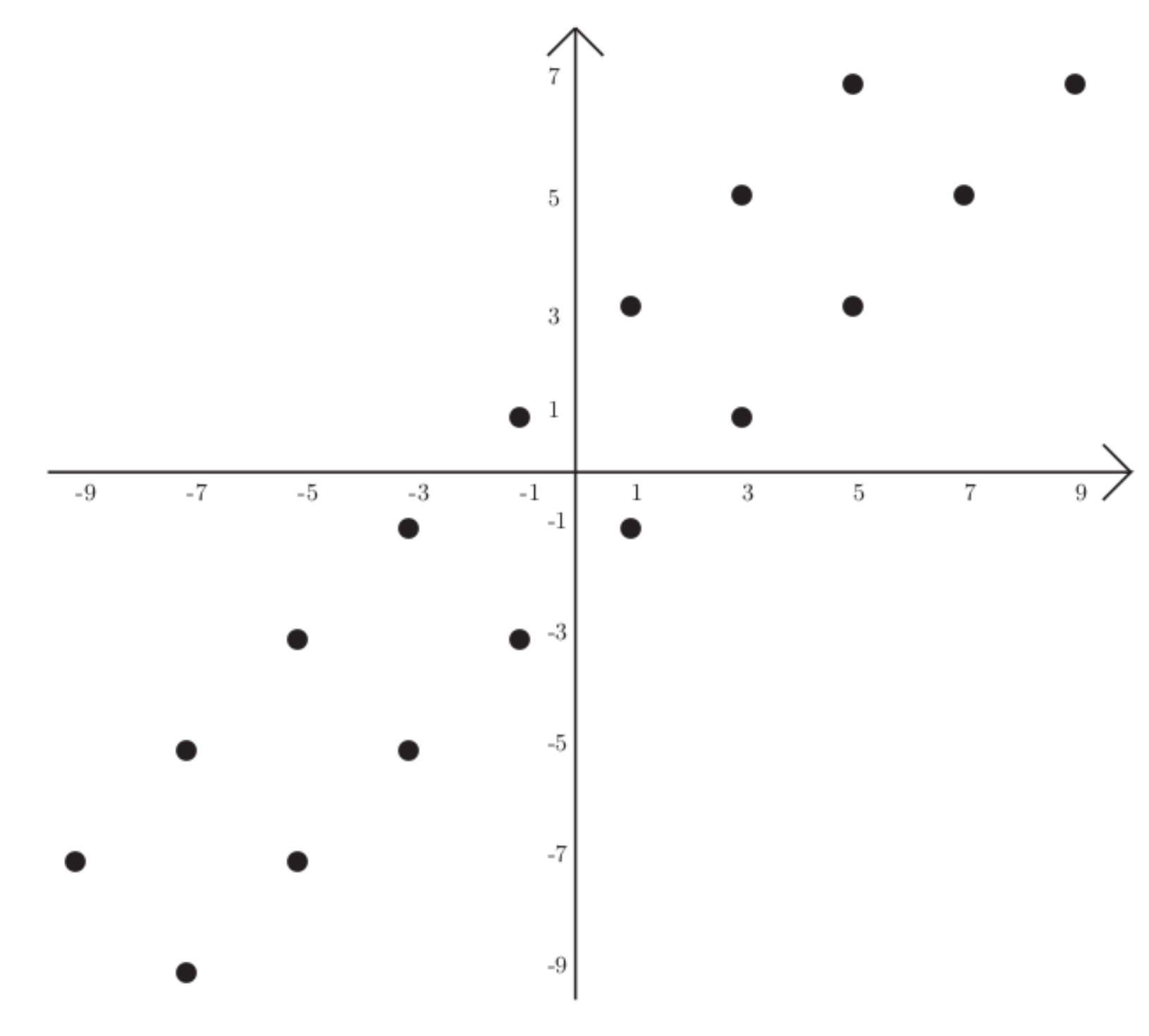}

$\mathcal{F}_\mathscr{B}=\{(2k+1,2k-1), (2k-1,2k+1)\ : \ k\in\Z\}$.
\end{center}
It follows immediately that $\mathbf{0}\in X_\eta$. Moreover, since $\Lambda_i\subseteq2\Z\times2\Z\cup(2\Z+1)\times(2\Z+1)$ for any $i\geq3$, we have $\Lambda_i+\Lambda_j\subseteq2\Z\times2\Z\cup(2\Z+1)\times(2\Z+1)$ and $\mathscr{B}$ does not contain an infinite pairwise coprime subset. 
\end{Example}
\begin{Remark}
The conditions in Theorem \ref{one-dim} are equivalent to
\begin{equation}\label{d}\tag{d'}
\text{ there exists an infinite pairwise coprime set }\{a_i\}_{i\geq1} \text{ such that }\bigcup_{i\geq1}a_i\Z\subseteq\mathcal{M}_\mathscr{B}. 
\end{equation}
Clearly \ref{(d)} $\implies$ \eqref{d}. Now, suppose that \eqref{d} holds. Then, for any $i\geq1$, there exists $j_i$ with $b_{j_i}\divides a_i$. Clearly, $\{b_{j_i}\}_{i\geq1}$ is again infinite and pairwise coprime. 

In case of lattices, \eqref{d} takes the following form:
\begin{equation}\label{d_lat}\tag{d'}
\text{ there exists an infinite pairwise coprime set }\{\widetilde{\Lambda}_i\}_{i\geq1} \text{ with }\bigcup_{i\geq1}\widetilde{\Lambda}_i\subseteq\mathcal{M}_\mathscr{B}.
\end{equation}
The following question arises:
\begin{Question}\label{que}
\begin{enumerate}[(A)]
\item Is \eqref{d_lat} equivalent to \eqref{ld} in Theorem \ref{general_lattices}?\label{QA}
\item Is \eqref{d_lat} equivalent to \eqref{la} in Theorem \ref{general_lattices}?\label{QB}
\end{enumerate}
\end{Question}
Notice that we have $\bigcup_{i\geq3}p_i\Z^2\subseteq\mathcal{M}_\mathscr{B}$ in Example \ref{ex2}. So \eqref{d_lat} holds but \eqref{ld} does not. Hence Question \ref{que} \eqref{QA} has the negative answer. 
\end{Remark}
To answer negatively Question \ref{que} \eqref{QB}, we will need the following:
\begin{Lemma}\label{rectan}
If $\{\Lambda_i=(a_i,b_i)\Z+(0,d_i)\Z\ : i\geq1\}$ is pairwise coprime then the projection of $(\{0\}\times\Z)\cap\mathcal{M}_\mathscr{B}$ onto the second coordinate contains an infinite pairwise coprime set.
\end{Lemma}
\begin{proof}
It suffices to notice that $(0,a_id_i)\in\Lambda_i$. Moreover, by Proposition \ref{prime}, $\{|a_id_i|\}_{i\geq1}$ contains an infinite subset of pairwise coprime integers.
\end{proof}
Now, we are ready to consider the following: 
\begin{Example}\label{ex1}
Let $\mathscr{B}=\{\Lambda_i\}_{i\geq1}$, where 
\begin{align*}
\Lambda_1&=(1,1)\Z+(0,2)\Z,\\
\Lambda_2&=\Z\times2\Z,\\
\Lambda_{2i+1}&=(2p_i,1)\Z+(0,2)\Z, i\geq1,\\
\Lambda_{2i+2}&=(2^{i+1},1)\Z+(0,2)\Z, i\geq1,
\end{align*}
where $\{p_i\}_{i\geq1}$ is the set of all odd primes. 
We have $\Z^2\setminus(\Lambda_1\cup\Lambda_2)=2\Z\times(2\Z+1)$. Take $(n,m)\in\Z^2$ such that $2\ndivides m$ and $n=2^ar$, where $a\geq1$ and $2\ndivides r$. Then
\begin{itemize}
\item $(n,m)\in\bigcup_{i\geq1}\Lambda_{2i+2}\iff a\geq2$, \item $(n,m)\in\bigcup_{i\geq1}\Lambda_{2i+1}\iff a=1\text{ and } p_i\divides n\text{ for some }i\geq1$.
\end{itemize}
Therefore
\begin{equation}\label{free}
\mathcal{F}_\mathscr{B}=\Z^2\setminus\mathcal{M}_\mathscr{B}=\{-2,0,2\}\times(2\Z+1).
\end{equation} 
\begin{center}
\includegraphics[scale=0.7]{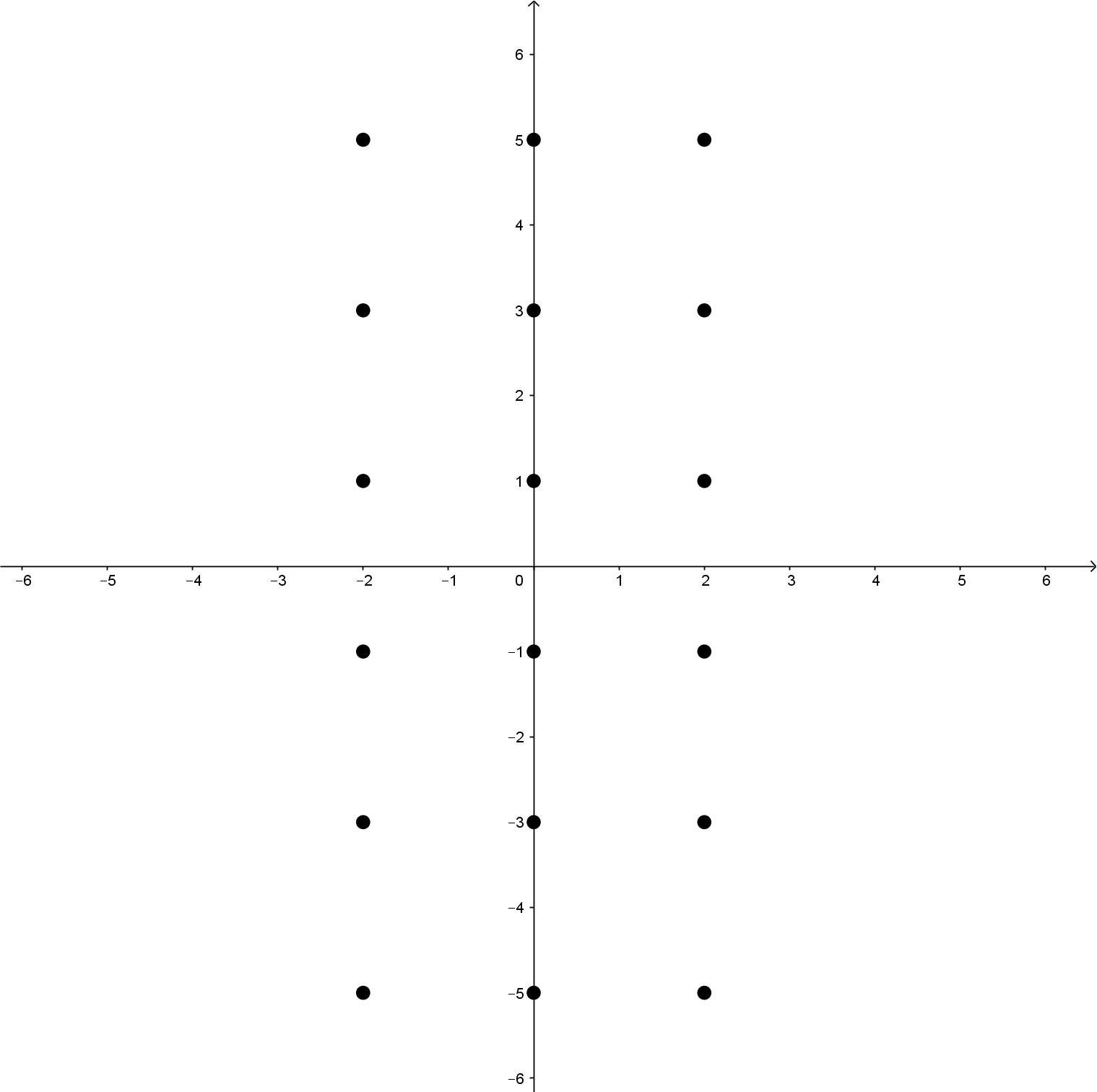}
\end{center}
It follows immediately that $\mathbf{0}\in X_\eta$, i.e., $(X_\eta,(S_\mathbf{n})_{\mathbf{n}\in\Z^2})$ is proximal.

Suppose now that
\begin{equation} \label{(*)}
\bigcup_{j\geq1}\widetilde{\Lambda}_j\subseteq\mathcal{M}_\mathscr{B},
\end{equation}
 with $\{\widetilde{\Lambda}_j\}_{j\geq1}$ pairwise coprime. By Lemma \ref{rectan}, the projection of $(\{0\}\times\Z)\cap\mathcal{M}_\mathscr{B}$ onto the second coordinate contains an infinite pairwise coprime set. However, $(\{0\}\times\Z)\cap\mathcal{M}_\mathscr{B}\subseteq\Z^2=\{0\}\times2\Z$, which is impossible in view of \eqref{(*)}. We conclude that \eqref{la} $\centernot\implies$ \eqref{d}.
\end{Example}
\begin{Remark}
Notice that for $(\Lambda_i)_{i\geq1}$ from Example \ref{ex1} we have $\bigcap_{i\geq1}\Lambda_i=\{(0,0)\}$. 
\end{Remark}
In the last example, we will show how to use Remark \ref{automorphism} to obtain an extension of Corollary \ref{rectangular}.
\begin{Example}
Let $m=2$, $k\in\Z$ and consider $\{(a_i,b_i)\Z+(0,d_i)\Z\}_{i\geq1}$, i.e., $(a_i,ka_i)\Z+(0,d_i)\Z=A(a_i\Z\times d_i\Z)$ for $A=\begin{pmatrix}1 & 0 \\ k & 1\end{pmatrix}$, $i\geq1$. By Theorem \ref{ideals_general} and Remark \ref{automorphism} the following are equivalent:
\begin{itemize}
\item $(X_\eta, (S_\mathbf{n})_{\mathbf{n}\in\Z^2})$ is proximal,
\item $(X_{\eta_A}, (S_\mathbf{n})_{\mathbf{n}\in\Z^2})$ is proximal,
\item $\{a_i\Z\times d_i\Z\}_{i\geq1}$ contains an infinite pairwise coprime set,
\item $\{(a_i,ka_i)\Z+(0,d_i)\Z\}_{i\geq1}$ contains an infinite pairwise coprime set
\end{itemize}
(the two latter conditions are equivalent as $A$ is a group isomorphism).
\end{Example}
We leave the following open:
\begin{Question} Can proximality of $\mathscr{B}$-free systems in general case of lattices be characterized by an arithmetic property of the family $\mathscr{B}=\{\Lambda_i\}_{i\geq1}$? By Example \ref{ex1}, such a property must be weaker than \eqref{d}.
\end{Question}

\bigskip
\footnotesize

\noindent
Aurelia Dymek\\
\textsc{Faculty of Mathematics and Computer Science, Nicolaus Copernicus University, Chopina 12/18, 87-100 Toru\'{n}, Poland}\par\nopagebreak
\noindent
\texttt{aurbart@mat.umk.pl}

\medskip
\end{document}